\documentclass[reqno, 12pt]{amsart}
\usepackage{amssymb}
\usepackage{amsfonts}
\usepackage{srcltx}

\newtheorem{theorem}{Theorem}[section]
\newtheorem{lemma}[theorem]{Lemma}

\newtheorem{corollary}[theorem]{Corollary}
\theoremstyle{definition}

\theoremstyle{remark}
\newtheorem*{remark}{Remark}

\DeclareMathOperator{\diam}{diam}

\begin{document}
\title[Asymptotically nonexpansive mappings]{The super fixed point property
for asymptotically nonexpansive mappings}
\author[]{Andrzej Wi\'{s}nicki}
\address{Andrzej Wi\'{s}nicki, Institute of Mathematics, Maria Curie-Sk{\l }%
odowska University, 20-031 Lublin, Poland}
\email{awisnic@hektor.umcs.lublin.pl}

\begin{abstract}
We show that the super fixed point property for nonexpansive mappings and
for asymptotically nonexpansive mappings in the intermediate sense are
equivalent. As a consequence, we obtain fixed point theorems for
asymptotically nonexpansive mappings in uniformly nonsquare and uniformly
noncreasy Banach spaces. The results are generalized for commuting families
of asymptotically nonexpansive mappings.
\end{abstract}

\maketitle

\section{Introduction}

The classical problem in metric fixed point theory, a branch of fixed point
theory which emerged from the Banach contraction principle, is concerned
with the existence of fixed points of nonexpansive mappings. Recall that a
mapping $T:C\rightarrow C$ is nonexpansive if
\begin{equation*}
\Vert Tx-Ty\Vert \leq \Vert x-y\Vert
\end{equation*}%
for all $x,y\in C$. A Banach space $X$ is said to have the fixed point
property (FPP for short) if every nonexpansive self-mapping defined on a
nonempty bounded closed and convex set $C\subset X$ has a fixed point (see
\cite{AyDoLo, GoKi, KiSi}). One of the natural and extensively studied
generalizations of nonexpansive mappings was introduced by Goebel and Kirk
\cite{GoKi1}. A mapping $T:C\rightarrow C$ is said to be asymptotically
nonexpansive if there exists a sequence of real numbers $(k_{n})$ with $%
\lim_{n}k_{n}=1$ such that%
\begin{equation*}
\Vert T^{n}x-T^{n}y\Vert \leq k_{n}\Vert x-y\Vert
\end{equation*}%
for all $x,y\in C$ and $n\in \mathbb{N}$.

Let $B$ be the closed unit ball in $\ell _{2}$ and set
\begin{equation*}
T(x_{1},x_{2},x_{3},...)=(0,x_{1}^{2},a_{2}x_{2},a_{3}x_{3},...),
\end{equation*}%
where $(x_{1},x_{2},x_{3},...)\in B$ and $(a_{n})$ is a sequence of reals in
$(0,1)$ such that $\prod_{n=2}^{\infty }a_{n}=\frac{1}{2}.$ Then
\begin{equation*}
\Vert Tx-Ty\Vert \leq 2\Vert x-y\Vert
\end{equation*}%
and
\begin{equation*}
\Vert T^{n}x-T^{n}y\Vert \leq 2\prod_{i=2}^{n}a_{i}\Vert x-y\Vert
\end{equation*}
(see \cite{GoKi1}). This shows that the class of asymptotically nonexpansive
mappings is wider than the class of nonexpansive mappings.

In spite of the common belief that asymptotically nonexpansive mappings
share a lot of properties of nonexpansive mappings, there exist relatively
few results concerning the existence of fixed points for such mappings. The
original result from \cite{GoKi1} stating that asymptotically nonexpansive
mappings have the fixed point property in uniformly convex spaces was
generalized in \cite{Xu} when $X$ is nearly uniformly convex, in \cite%
{LiTaXu} when $X$ satisfies the uniform Opial condition and in \cite{KiXu}
when $X$ has uniform normal structure. It is still unknown whether normal
structure implies the fixed point property for asymptotically nonexpansive
mappings acting on a convex and weakly compact subset of a Banach space $X.$
Until now, the situation has been even worse in Banach spaces without normal
structure.

In 1998, Kirk, Martinez Ya\~{n}ez and Shin \cite{KiMaSh} showed that if $X$
has the super fixed point property for nonexpansive mappings (i.e., every
Banach space finitely representable in $X$ has FPP), then every
asymptotically nonexpansive mapping defined on a bounded closed and convex
subset of $X$ has approximate fixed points, i.e., there exists a sequence $%
(x_{n})$ such that $\lim_{n}\Vert Tx_{n}-x_{n}\Vert =0$. In the present
paper we strengthen this result by showing, in Theorem \ref{Th2}, that the
super fixed point property for nonexpansive mappings is equivalent to the
super fixed point property for asymptotically nonexpansive mappings in the
intermediate sense (see Section 2 for the definition). In particular, we
obtain fixed point theorems for asymptotically nonexpansive mappings in both
uniformly nonsquare and uniformly noncreasy Banach spaces. In Section 3, the
above results are extended for commuting families of asymptotically
nonexpansive mappings in the intermediate sense.

It was shown in \cite[Th. 10]{DoLeTu} that every Banach space $X$ which
contains an isomorphic copy of $c_{0}$ fails the fixed point property for
asymptotically nonexpansive mappings. Our results support the conjecture
that the fixed point property for nonexpansive mappings and for
asymptotically nonexpansive mappings are equivalent which would imply the
failure of the FPP inside isomorphic copies of $c_{0}.$


\section{Main result}

Let $X$ and $Y$ be Banach spaces and let $0<\varepsilon <1.$ A linear map $%
T:Y\rightarrow X$ is an $\varepsilon $-isometry if%
\begin{equation*}
\left( 1-\varepsilon \right) \left\Vert y\right\Vert \leq \left\Vert
Ty\right\Vert \leq \left( 1+\varepsilon \right) \left\Vert y\right\Vert
\end{equation*}%
for all $y\in Y$. Recall that $Y$ is said to be finitely representable in $X$
if for each $\varepsilon \in \left( 0,1\right) $ and every finite
dimensional subspace $M\subset Y$ there exists an $\varepsilon $-isometry $%
T:M\rightarrow X$.

We say that $X$ is superreflexive if every Banach space $Y$ which is
finitely representable in $X$ is reflexive. A Banach space $X$ has the super
fixed point property for nonexpansive mappings (SFPP) if every Banach space $%
Y$ which is finitely representable in $X$ has FPP. It follows from the
result of van Dulst and Pach \cite[Th. 3.2]{DuPa} that SFPP implies
superreflexivity.

The notion of finite representability is closely related with the
construction of the Banach space ultrapower. Let $\mathcal{U}$ be an
ultrafilter defined on a set $I$. The ultrapower $\widetilde{X}$ (or $(X)_{%
\mathcal{U}}$) of a Banach space $X$ is the quotient space of
\begin{equation*}
l_{\infty }(X)=\left\{ (x_{n}):x_{n}\in X\text{\ for all }n\in I\text{ and }%
\left\Vert (x_{n})\right\Vert =\sup_{n}\left\Vert x_{n}\right\Vert <\infty
\right\}
\end{equation*}%
by
\begin{equation*}
\left\{ (x_{n})\in l_{\infty }(X):\lim\limits_{n\rightarrow \mathcal{U}%
}\left\Vert x_{n}\right\Vert =0\right\} .
\end{equation*}%
Here $\lim\limits_{n\rightarrow \mathcal{U}}$ denotes the ultralimit over $%
\mathcal{U}$. One can prove that the quotient norm on $(X)_{\mathcal{U}}$ is
given by
\begin{equation*}
\left\Vert (x_{n})_{\mathcal{U}}\right\Vert =\lim\limits_{n\rightarrow
\mathcal{U}}\left\Vert x_{n}\right\Vert ,
\end{equation*}%
where $(x_{n})_{\mathcal{U}}$ is the equivalence class of $(x_{n})$. It is
also clear that $X$ is isometric to a subspace of $(X)_{\mathcal{U}}$ by the
mapping $x\rightarrow (x)_{\mathcal{U}}$.

The connection between ultrapowers and finite representability was observed
independently by Henson and Moore \cite{HeMo}, and Stern \cite{St} (see also
\cite{AkKh, He, Si}).

\begin{theorem}
\label{finitely}A Banach space $Y$ is finitely representable in $X$ if and
only if there exists an ultrafilter $\mathcal{U}$ such that $Y$ is isometric
to a subspace of $(X)_{\mathcal{U}}$.
\end{theorem}

It follows from the above theorem that $X$ has SFPP iff every ultrapower $%
(X)_{\mathcal{U}}$ has FPP.

In 1980, the Banach space ultrapower construction was applied in fixed point
theory by Maurey \cite{Ma} who proved the fixed point property for all
reflexive subspaces of $L_{1}[0,1]$ and the weak fixed point property for $%
c_{0}$ and $H^{1}$. Inspired by \cite{KiMaSh}, we apply this construction to
asymptotically nonexpansive mappings in a slightly more general setting.
Recall that a mapping \thinspace $T:C\rightarrow C$ is said to be
asymptotically nonexpansive in the intermediate sense if $T$ is continuous
and%
\begin{equation}
\limsup_{n\rightarrow \infty }\sup_{x,y\in C}(\Vert T^{n}x-T^{n}y\Vert
-\Vert x-y\Vert )\leq 0  \label{inter}
\end{equation}%
(in the original definition, in\ \cite{BrKuRe}, $T$ was assumed to be
uniformly continuous). In particular, the condition (\ref{inter}) is
satisfied if $\limsup_{n\rightarrow \infty }|T^{n}|\leq 1$, where $|T^{n}|$
denotes the (exact) Lipschitz constant of $T^{n}$ (and $C$ is bounded).

Let $C$ be a nonempty bounded closed and convex subset of a Banach space $X$
and $T:C\rightarrow C$ be asymptotically nonexpansive in the intermediate
sense. Take a free ultrafilter $p$ on $\mathbb{N}$ and denote by $\widetilde{%
C}\subset (X)_{p}$ the set
\begin{equation*}
\widetilde{C}=\left\{ (x_{n})_{p}\in (X)_{p}:x_{n}\in C\text{\ for all }n\in
\mathbb{N}\right\} .
\end{equation*}%
Let $(\mathcal{N},\preceq )$ be a directed set, where
\begin{equation*}
\mathcal{N}=\left\{ (\alpha _{n})\in \mathbb{N}^{\mathbb{N}}:\alpha
_{0}<\alpha _{1}<...<\alpha _{n}<...\text{ }\right\}
\end{equation*}%
is a family of all increasing sequences of natural numbers directed by the
relation $(\alpha _{n})\preceq (\beta _{n})$ iff $\alpha _{n}\leq \beta _{n}$
for every $n\in \mathbb{N}$. Notice that if $(x_{n})_{p},(y_{n})_{p}\in
\widetilde{C}$ and $(\alpha _{n})\in \mathcal{N},$ then
\begin{equation*}
\lim_{n\rightarrow p}(\Vert T^{\alpha _{n}}x_{n}-T^{\alpha _{n}}y_{n}\Vert
-\Vert x_{n}-y_{n}\Vert )\leq \limsup_{n\rightarrow \infty }\sup_{x,y\in
C}(\Vert T^{n}x-T^{n}y\Vert -\Vert x-y\Vert )\leq 0.
\end{equation*}%
Therefore, we may extend the mapping $T$ by setting, unambiguosly,%
\begin{equation}
\widehat{T}_{(\alpha _{n})}(x_{n})_{p}=(T^{\alpha _{n}}x_{n})_{p}.
\label{def1}
\end{equation}%
It is not difficult to see that $\widehat{T}_{(\alpha _{n})}:\widetilde{C}%
\rightarrow \widetilde{C}$ is nonexpansive for every $(\alpha _{n})\in
\mathcal{N}$. For $x\in C,$ we shall write $\dot{x}=(x)_{p}=(x,x,...)_{p}.$

\begin{lemma}
\label{Lem1}Let $T:C\rightarrow C$ be asymptotically nonexpansive\ in the
intermediate sense and suppose that there exists $\widetilde{y}\in
\widetilde{C}$ such that
\begin{equation}
\widehat{T}_{(\alpha _{n})}\widetilde{y}=\widetilde{y}  \label{contr}
\end{equation}%
for all $(\alpha _{n})\in \mathcal{N}.$ Let $\left\Vert \widetilde{y}-\dot{x}%
_{0}\right\Vert <\delta $ for some $x_{0}\in C$ and $\delta >0$. Then, for
every $\varepsilon >0$ there exist $x\in C$ and $n_{0}\in \mathbb{N}$ such
that $\left\Vert x-x_{0}\right\Vert <\delta $ and $\Vert T^{n}x-x\Vert
<\varepsilon $ for every $n\geq n_{0}.$
\end{lemma}

\begin{proof}
Since $\left\Vert \widetilde{y}-\dot{x}_{0}\right\Vert <\delta ,$ there
exists a sequence $(y_{n})$ in $C$ such that $\left\Vert
y_{n}-x_{0}\right\Vert <\delta $ for all $n\in \mathbb{N}$ and $\widetilde{y}%
=$ $(y_{n})_{p}.$ Assume, conversely to our claim, that there exists $%
\varepsilon _{0}>0$ such that for every $x\in C$ and $n_{0}\in \mathbb{N}$
there exists $n\geq n_{0}$ such that $\left\Vert x-x_{0}\right\Vert \geq
\delta $ or $\Vert T^{n}x-x\Vert \geq \varepsilon _{0}.$ We shall define a
sequence $(\beta _{n})$ by induction. For $n=0$ and $y_{0}\in C$, there
exists $\beta _{0}$ such that $\Vert T^{\beta _{0}}y_{0}-y_{0}\Vert \geq
\varepsilon _{0}$. Suppose that we have chosen\ $\beta _{0}<\beta
_{1}<...<\beta _{n}$ such that $\Vert T^{\beta _{i}}y_{i}-y_{i}\Vert \geq
\varepsilon _{0}$ for $i=0,1,...,n.$ By assumption, since $\left\Vert
y_{n+1}-x_{0}\right\Vert <\delta $, there exists $\beta _{n+1}>\beta _{n}$
such that $\Vert T^{\beta _{n+1}}y_{n+1}-y_{n+1}\Vert \geq \varepsilon _{0}$%
. (To be more precise, we can define, for example, $\beta _{n+1}$ as the
minimum of $\{\beta >\beta _{n}:\Vert T^{\beta }y_{n+1}-y_{n+1}\Vert \geq
\varepsilon _{0}\}$). Thus we obtain a sequence $(\beta _{n})\in \mathcal{N}$
such that $\Vert T^{\beta _{n}}y_{n}-y_{n}\Vert \geq \varepsilon _{0}$ for
all $n\in \mathbb{N}.$ Hence $\left\Vert \widehat{T}_{(\beta _{n})}%
\widetilde{y}-\widetilde{y}\right\Vert \geq \varepsilon _{0}$, a
contradiction with (\ref{contr}).
\end{proof}

A Banach space $X$ is said to have the fixed point property for
asymptotically nonexpansive mappings (in the intermediate sense) if every
asymptotically nonexpansive (in the intermediate sense) self-mapping acting
on a nonempty bounded closed and convex set $C\subset X$ has a fixed point.

\begin{theorem}
\label{Th1}Assume that $X$ has the super fixed point property for
nonexpansive mappings. Then $X$ has the fixed point property for
asymptotically nonexpansive mappings in the intermediate sense.
\end{theorem}

\begin{proof}
Assume that $X$ has the super fixed point property for nonexpansive
mappings. Let $T:C\rightarrow C$ be an asymptotically nonexpansive mapping
in the intermediate sense acting on a nonempty bounded closed and convex set
$C\subset X.$ By \cite[Th. 3.2]{DuPa}, $X$ is superreflexive and hence $C$
is weakly compact. Without loss of generality we can assume that $\diam C=1.$
Take a free ultrafilter $p$ on $\mathbb{N},$ $(\alpha _{n})\in \mathcal{N},$
and define $\widehat{T}_{(\alpha _{n})}$ by (\ref{def1}). Notice that for
every $(\alpha _{n}),(\beta _{n})\in \mathcal{N}$ and any $(z_{n})_{p}\in $ $%
\widetilde{C}$,
\begin{equation*}
(\widehat{T}_{(\alpha _{n})}\circ \widehat{T}_{(\beta
_{n})})(z_{n})_{p}=(T^{\alpha _{n}}T^{\beta _{n}}z_{n})_{p}=(\widehat{T}%
_{(\beta _{n})}\circ \widehat{T}_{(\alpha _{n})})(z_{n})_{p}.
\end{equation*}%
It follows from the Bruck theorem (see \cite[Th. 1]{Br2}) that there exists $%
\widetilde{y}_{0}\in \widetilde{C}$ such that $\widehat{T}_{(\alpha _{n})}%
\widetilde{y}_{0}=\widetilde{y}_{0}$ for all $(\alpha _{n})\in \mathcal{N}$
(a similar argument but for two mappings was used in \cite[Th. 4.1]{KiMaSh}%
). Fix $\varepsilon <1$ and $x_{0}\in C.$ We shall define by induction a
sequence $(n_{j})$ of natural numbers and a sequence $(x_{j})$ of elements
in $C$ such that
\begin{equation}
\left\Vert x_{j}-x_{j-1}\right\Vert <3\varepsilon ^{j-1}\text{ and }\Vert
T^{n}x_{j}-x_{j}\Vert <\varepsilon ^{j}\text{ for every }n\geq n_{j},j\geq 1.
\label{seq1}
\end{equation}%
By Lemma \ref{Lem1}, there exist $x_{1}\in C$ and $n_{1}\in \mathbb{N}$ such
that $\Vert T^{n}x_{1}-x_{1}\Vert <\varepsilon $ for every $n\geq n_{1}$ and
$\left\Vert x_{1}-x_{0}\right\Vert \leq \diam C<3$.

Suppose that we have chosen\ natural numbers $n_{1},...,n_{j}$ and $%
x_{1},...,x_{j}\in C$ $(j\geq 1)$ such that%
\begin{equation*}
\left\Vert x_{i}-x_{i-1}\right\Vert <3\varepsilon ^{i-1}\text{ and }\Vert
T^{n}x_{i}-x_{i}\Vert <\varepsilon ^{i}\text{ for every }n\geq n_{i},1\leq
i\leq j.
\end{equation*}%
Let
\begin{equation*}
D_{j}=\left\{ \widetilde{y}=(y_{n})_{p}\in \widetilde{C}:\limsup_{(\alpha
_{n})\in \mathcal{N}}\Vert \widehat{T}_{(\alpha _{n})}\dot{x}_{j}-\widetilde{%
y}\Vert \leq \varepsilon ^{j}\right\} ,
\end{equation*}%
where $\dot{x}_{j}=(x_{j},x_{j},...)_{p}$ and%
\begin{equation*}
\limsup_{(\alpha _{n})\in \mathcal{N}}\Vert \widehat{T}_{(\alpha _{n})}\dot{x%
}_{j}-\widetilde{y}\Vert =\inf_{(\alpha _{n})\in \mathcal{N}}\sup_{(\beta
_{n})\succeq (\alpha _{n})}\lim_{n\rightarrow p}\left\Vert T^{\beta
_{n}}x_{j}-y_{n}\right\Vert
\end{equation*}%
denotes the upper limit of the net $(\Vert \widehat{T}_{(\alpha _{n})}\dot{x}%
_{j}-\widetilde{y}\Vert )_{(\alpha _{n})\in \mathcal{N}}.$ It is not
difficult to see that $D_{j}$ is a nonempty closed and convex subset of $%
\widetilde{C}$ (notice that $\dot{x}_{j}\in D_{j}$). Futhermore, for a fixed
$(\beta _{n})\in \mathcal{N}$ and $\widetilde{y}\in D_{j},$%
\begin{align*}
& \limsup_{(\alpha _{n})\in \mathcal{N}}\Vert \widehat{T}_{(\alpha _{n})}%
\dot{x}_{j}-\widehat{T}_{(\beta _{n})}\widetilde{y}\Vert =\limsup_{(\alpha
_{n})\in \mathcal{N}}\Vert \widehat{T}_{(\alpha _{n}+\beta _{n})}\dot{x}_{j}-%
\widehat{T}_{(\beta _{n})}\widetilde{y}\Vert \\
& \leq \limsup_{(\alpha _{n})\in \mathcal{N}}\Vert \widehat{T}_{(\alpha
_{n})}\dot{x}_{j}-\widetilde{y}\Vert \leq \varepsilon ^{j},
\end{align*}%
and hence $\widehat{T}_{(\beta _{n})}(D_{j})\subset D_{j}$ for every $(\beta
_{n})\in \mathcal{N}.$ Again, by Bruck's theorem, there exists $\widetilde{y}%
_{j}\in D_{j}$ such that $\widehat{T}_{(\alpha _{n})}\widetilde{y}_{j}=%
\widetilde{y}_{j}$ for all $(\alpha _{n})\in \mathcal{N}.$ Notice that $%
\left\Vert \widetilde{y}_{j}-\dot{x}_{j}\right\Vert \leq 2\varepsilon
^{j}<3\varepsilon ^{j}$ and by Lemma \ref{Lem1}, there exist $x_{j+1}\in C$
and $n_{j+1}\in \mathbb{N}$ such that $\left\Vert x_{j+1}-x_{j}\right\Vert
<3\varepsilon ^{j}$ and $\Vert T^{n}x_{j+1}-x_{j+1}\Vert <\varepsilon ^{j+1}$
for every $n\geq n_{j+1}.$

Thus we obtain by induction a sequence $(n_{j})$ of natural numbers and a
sequence $(x_{j})$ of elements in $C$ such that (\ref{seq1}) is satisfied.
It follows that $(x_{j})$ is a Cauchy sequence converging to some $x\in C$.
Hence%
\begin{align*}
& \Vert T^{n}x-x\Vert \leq \Vert T^{n}x-T^{n}x_{j}\Vert +\Vert
T^{n}x_{j}-x_{j}\Vert +\Vert x_{j}-x\Vert \\
& \leq (\Vert T^{n}x-T^{n}x_{j}\Vert -\Vert x_{j}-x\Vert )+\varepsilon
^{j}+2\Vert x_{j}-x\Vert
\end{align*}%
for every $n\geq n_{j},j\geq 1$ and, consequently, $\lim_{n\rightarrow
\infty }$ $\Vert T^{n}x-x\Vert =0.$ Since $T$ is continuous, $Tx=x$.
\end{proof}

A Banach space $X$ is said to have the super fixed point property for
asymptotically nonexpansive mappings (in the intermediate sense) if every
Banach space $Y$ which is finitely representable in $X$ has the fixed point
property for asymptotically nonexpansive mappings (in the intermediate
sense). We can strengthen Theorem \ref{Th1} in the following way.

\begin{theorem}
\label{Th2}A Banach space $X$ has the super fixed point property for
nonexpansive mappings if and only if $X$ has the super fixed point property
for asymptotically nonexpansive mappings in the intermediate sense.
\end{theorem}

\begin{proof}
Assume that $X$ has SFPP for nonexpansive mappings and let $\mathcal{U}$ be
an ultrafilter defined on a set $I.$ By Theorem \ref{finitely}, $(X)_{%
\mathcal{U}}$ has SFPP, too, and it follows from Theorem \ref{Th1} that $%
(X)_{\mathcal{U}}$ has the fixed point property for asymptotically
nonexpansive mappings in the intermediate sense. By Theorem \ref{finitely}
again, $X$ has the super fixed point property for asymptotically
nonexpansive mappings in the intermediate sense. The reverse implication is
obvious.
\end{proof}

We conclude this section by giving some consequences of Theorem \ref{Th1}.
Recall \cite{Ja} that a Banach space is uniformly nonsquare if
\begin{equation*}
\sup_{x,y\in S_{X}}\min \left\{ \left\Vert x+y\right\Vert ,\left\Vert
x-y\right\Vert \right\} <2.
\end{equation*}%
Garc\'{\i}a Falset, Llor\'{e}ns Fuster and Mazcu\~{n}an Navarro (see \cite%
{GLM2}) solved a long-standing problem in metric fixed point theory by
proving that uniformly nonsquare Banach spaces have FPP and, in consequence,
SFPP for nonexpansive mappings.

\begin{corollary}
Let $C$ be a nonempty bounded closed and convex subset of a uniformly
nonsquare Banach space. Then every asymptotically nonexpansive in the
intermediate sense mapping $T:C\rightarrow C$ has a fixed point.
\end{corollary}

In \cite{Pr}, Prus introduced the notion of uniformly noncreasy spaces. A
real Banach space $X$ is said to be uniformly noncreasy if for every $%
\varepsilon >0$ there is $\delta >0$ such that if $f,g\in S_{X^{\ast }}$ and
$\Vert f-g\Vert \geq \varepsilon $, then $\diam S(f,g,\delta )\leq
\varepsilon $, where
\begin{equation*}
S(f,g,\delta )=\left\{ x\in B_{X}:f(x)\geq 1-\delta \text{\ }\wedge
\;g(x)\geq 1-\delta \right\}
\end{equation*}%
($\diam\emptyset =0)$. It is known that both uniformly convex and uniformly
smooth spaces are uniformly noncreasy. The Bynum space $l^{2,\infty }$,
which is $l^{2}$ space with the norm
\begin{equation*}
\Vert x\Vert _{2,\infty }=\max \left\{ \Vert x^{+}\Vert _{2},\Vert
x^{-}\Vert _{2}\right\} ,
\end{equation*}%
and the space $X_{\sqrt{2}}$, which is $l^{2}$ space with the norm
\begin{equation*}
\Vert x\Vert _{\sqrt{2}}=\max \left\{ \Vert x\Vert _{2},\sqrt{2}\Vert x\Vert
_{\infty }\right\} ,
\end{equation*}%
are examples of uniformly noncreasy spaces without normal structure. It was
proved in \cite{Pr} that all uniformly noncreasy spaces are superreflexive
and have SFPP. This yields

\begin{corollary}
Let $C$ be a nonempty bounded closed and convex subset of a uniformly
noncreasy Banach space. Then every asymptotically nonexpansive in the
intermediate sense mapping $T:C\rightarrow C$ has a fixed point.
\end{corollary}

Recently, a fixed point theorem in direct sums of two Banach spaces was
proved in \cite{Wi}. Assume that $X$ has SFPP (for nonexpansive mappings)
and $Y$ is uniformly convex, uniformly smooth or finite dimensional. Since
uniformly convex, uniformly smooth as well as finite dimensional spaces are
stable under passing to the Banach space ultrapowers and have uniform normal
structure, it follows from \cite[Th. 3.4]{Wi}, that $X\oplus Y$ with a
strictly monotone norm has SFPP. Thus we obtain the following theorem.

\begin{corollary}
Assume that $X$ has SFPP and $Y$ is uniformly convex, uniformly smooth or
finite dimensional. Then $X\oplus Y$ with a strictly monotone norm has the
fixed point property for asymptotically nonexpansive mappings in the
intermediate sense.
\end{corollary}

\section{Common fixed points}

In this section we generalize Theorem \ref{Th1} for a commuting family of
mappings. Let $\{T_{t}:t\in T\}$ be a commuting family of asymptotically
nonexpansive self-mappings in the intermediate sense acting on a nonempty
bounded closed and convex subset $C$ of a Banach space $X.$ Consider the set%
\begin{align*}
\mathcal{A}=& \left\{ \{(t_{1},\alpha _{1}),(t_{2},\alpha
_{2}),...,(t_{k},\alpha _{k})\}:t_{1},...,t_{k}\in T,\,t_{i}\neq t_{j}\text{
for }i\neq j,\right. \\
& \ \ \left. \alpha _{1},\alpha _{2},...,\alpha _{k}\in \mathbb{N}%
,k>0\right\} ,
\end{align*}%
directed by the relation%
\begin{equation*}
\{(t_{1},\alpha _{1}),(t_{2},\alpha _{2}),...,(t_{k},\alpha
_{k})\}\sqsubseteq \{(s_{1},\beta _{1}),(s_{2},\beta _{2}),...,(s_{m},\beta
_{m})\}
\end{equation*}%
iff
\begin{equation*}
\{t_{1},t_{2},...,t_{k}\}\subseteq \{s_{1},s_{2},...s_{m}\}\text{ and }%
\forall i~\forall j\ (t_{i}=s_{j}\Rightarrow \alpha _{i}\leq \beta _{j}).
\end{equation*}%
If $v=\{(t_{1},\alpha _{1}),(t_{2},\alpha _{2}),...,(t_{k},\alpha _{k})\}\in
\mathcal{A},$ write%
\begin{equation*}
T_{v}x=T_{t_{1}}^{\alpha _{1}}T_{t_{2}}^{\alpha _{2}}...T_{t_{k}}^{\alpha
_{k}}x,
\end{equation*}%
and let%
\begin{equation*}
\mathcal{D}=\left\{ (v_{n})\in \mathcal{A}^{\mathbb{N}}:\limsup_{n%
\rightarrow \infty }\sup_{x,y\in C}(\Vert T_{v_{n}}x-T_{v_{n}}y\Vert -\Vert
x-y\Vert )\leq 0\right\} .
\end{equation*}%
Note that $\mathcal{D}\neq \emptyset $ since $(\{(t,n)\})_{n\in \mathbb{N}%
}\in \mathcal{D}$ for every $t\in T.$ If $(v_{n}),(u_{n})\in \mathcal{D},$
define the relation $(v_{n})\preceq (u_{n})$ iff $v_{n}\sqsubseteq u_{n}$
for every $n\in \mathbb{N}$. It is not difficult to see that for every $%
(v_{n}),(u_{n})\in \mathcal{D}$ there exists $(w_{n})\in \mathcal{D}$ such
that $(v_{n})\preceq (w_{n})$ and $(u_{n})\preceq (w_{n}).$ Indeed, let%
\begin{equation*}
v_{n}=\{(t_{1}^{(n)},\alpha _{1}^{(n)}),(t_{2}^{(n)},\alpha
_{2}^{(n)}),...,(t_{k_{n}}^{(n)},\alpha _{k_{n}}^{(n)})\},
\end{equation*}%
\begin{equation*}
u_{n}=\{(s_{1}^{(n)},\beta _{1}^{(n)}),(s_{2}^{(n)},\beta
_{2}^{(n)}),...,(s_{m_{n}}^{(n)},\beta _{m_{n}}^{(n)})\},
\end{equation*}%
and put
\begin{equation}
w_{n}=\{(t_{1}^{(n)},\alpha _{1}^{(n)}),(t_{2}^{(n)},\alpha
_{2}^{(n)}),...,(t_{k_{n}}^{(n)},\alpha _{k_{n}}^{(n)}),(s_{1}^{(n)},\beta
_{1}^{(n)}),(s_{2}^{(n)},\beta _{2}^{(n)}),...,(s_{m_{n}}^{(n)},\beta
_{m_{n}}^{(n)})\},  \label{w_n}
\end{equation}%
$n\in \mathbb{N}$ (to shorten notation, we use the convention that if $%
t_{i}=s_{j}$ for some $i,j,$ then the pairs $(t_{i},\alpha
_{i}),(s_{j},\beta _{j})$ in $w_{n}$ are identified with one pair $%
(t_{i},\alpha _{i}+\beta _{j})$). Notice that%
\begin{align*}
s((T_{w_{n}}))& =\limsup_{n\rightarrow \infty }\sup_{x,y\in C}(\Vert
T_{v_{n}}T_{u_{n}}x-T_{v_{n}}T_{u_{n}}y\Vert -\Vert x-y\Vert ) \\
& \leq s((T_{v_{n}}))+s((T_{u_{n}}))\leq 0,
\end{align*}%
where%
\begin{equation*}
s((T_{v_{n}}))=\limsup_{n\rightarrow \infty }\sup_{x,y\in C}(\Vert
T_{v_{n}}x-T_{v_{n}}y\Vert -\Vert x-y\Vert ).
\end{equation*}%
Hence $(w_{n})\in \mathcal{D}$ and, clearly, $(v_{n})\preceq (w_{n})$ and $%
(u_{n})\preceq (w_{n}).$ Thus $(\mathcal{D},\preceq )$ is a directed set.

Let $p$ be a free ultrafilter on $\mathbb{N}.$ Then, for every $%
(x_{n})_{p},(y_{n})_{p}\in \widetilde{C}$ and $(v_{n})\in \mathcal{D},$%
\begin{equation}
\lim_{n\rightarrow p}(\Vert T_{v_{n}}x_{n}-T_{v_{n}}y_{n}\Vert -\Vert
x_{n}-y_{n}\Vert )\leq \limsup_{n\rightarrow \infty }\sup_{x,y\in C}(\Vert
T_{v_{n}}x-T_{v_{n}}y\Vert -\Vert x-y\Vert )\leq 0.  \label{estim}
\end{equation}%
Therefore, we may define unambiguously a mapping $\widehat{T}_{(v_{n})}:%
\widetilde{C}\rightarrow \widetilde{C},$ by setting%
\begin{equation}
\widehat{T}_{(v_{n})}(x_{n})_{p}=(T_{v_{n}}x_{n})_{p}.  \label{def2}
\end{equation}%
It follows from (\ref{estim}) that $\widehat{T}_{(v_{n})}$ is nonexpansive
for every $(v_{n})\in \mathcal{D}$.

We can now prove a counterpart of Lemma \ref{Lem1}.

\begin{lemma}
\label{Lem2}Let $\{T_{t}:t\in T\}$ be a commuting family of asymptotically
nonexpansive mappings in the intermediate sense acting on a nonempty bounded
closed and convex subset $C$ of a Banach space $X.$\ Suppose that there
exists $\widetilde{y}\in \widetilde{C}$ such that
\begin{equation}
\widehat{T}_{(v_{n})}\widetilde{y}=\widetilde{y}  \label{contr1}
\end{equation}%
for all $(v_{n})\in \mathcal{D}.$ Let $\left\Vert \widetilde{y}-\dot{x}%
_{0}\right\Vert <\delta $ for some $x_{0}\in C$ and $\delta >0$. Then, for
every $\varepsilon >0$ there exist $x\in C$ and $n\in \mathbb{N}$ such that $%
\left\Vert x-x_{0}\right\Vert <\delta $ and $\Vert T_{u}x-x\Vert
<\varepsilon $ for every $u\in \mathcal{D}^{\prime }(n),$ where%
\begin{align*}
\mathcal{D}^{\prime }(n)& =\left\{ v=\{(t_{1},\alpha _{1}),(t_{2},\alpha
_{2}),...,(t_{k},\alpha _{k})\}\in \mathcal{A}:\right.  \\
& \ \ \ \ \sup_{x,y\in C}(\Vert T_{v}x-T_{v}y\Vert -\Vert x-y\Vert )\leq
\frac{1}{n+1}\}.
\end{align*}
\end{lemma}

\begin{proof}
Since $\left\Vert \widetilde{y}-\dot{x}_{0}\right\Vert <\delta ,$ there
exists a sequence $(y_{n})$ in $C$ such that $\left\Vert
y_{n}-x_{0}\right\Vert <\delta $ for all $n\in \mathbb{N}$ and $\widetilde{y}%
=$ $(y_{n})_{p}.$ Assume, conversely to our claim, that there exists $%
\varepsilon _{0}>0$ such that for every $x\in C$ and $n\in \mathbb{N}$ there
exists $u\in \mathcal{D}^{\prime }(n)$ such that $\left\Vert
x-x_{0}\right\Vert \geq \delta $ or $\Vert T_{u}x-x\Vert \geq \varepsilon
_{0}.$ We shall define a sequence $(u_{n})\in \mathcal{D}$ by induction. For
$n=0$ and $y_{0}\in C$, there exists $u_{0}\in \mathcal{D}^{\prime }(0)$
such that $\Vert T_{u_{0}}y_{0}-y_{0}\Vert \geq \varepsilon _{0}$. Suppose
that we have chosen\ $u_{0}\in \mathcal{D}^{\prime }(0),u_{1}\in \mathcal{D}%
^{\prime }(1),...,u_{n}\in \mathcal{D}^{\prime }(n)$ such that $\Vert
T_{u_{i}}y_{i}-y_{i}\Vert \geq \varepsilon _{0}$ for $i=0,1,...,n.$ By
assumption, since $\left\Vert y_{n+1}-x_{0}\right\Vert <\delta $, there
exists $u_{n+1}\in \mathcal{D}^{\prime }(n+1)$ such that $\Vert
T_{u_{n+1}}y_{n+1}-y_{n+1}\Vert \geq \varepsilon _{0}$. Thus we obtain a
sequence $(u_{n})\in \mathcal{A}^{\mathbb{N}}$ such that $\Vert
T_{u_{n}}y_{n}-y_{n}\Vert \geq \varepsilon _{0}$ and
\begin{equation*}
\sup_{x,y\in C}(\Vert T_{u_{n}}x-T_{u_{n}}y\Vert -\Vert x-y\Vert )\leq \frac{%
1}{n+1}
\end{equation*}%
for all $n\in \mathbb{N}.$ Hence%
\begin{equation*}
\limsup_{n\rightarrow \infty }\sup_{x,y\in C}(\Vert
T_{u_{n}}x-T_{u_{n}}y\Vert -\Vert x-y\Vert )=0,
\end{equation*}%
i. e., $(u_{n})\in \mathcal{D}.$ But this contradicts (\ref{contr1}), since $%
\left\Vert \widehat{T}_{(u_{n})}\widetilde{y}-\widetilde{y}\right\Vert \geq
\varepsilon _{0}.$
\end{proof}

We will also make use of the following simple observation.

\begin{lemma}
\label{Lem3}For every $(u_{n})\in \mathcal{D}$ and $\tilde{x},\tilde{y}\in
\widetilde{C},$
\begin{equation*}
\limsup_{(v_{n})\in \mathcal{D}}\Vert \widehat{T}_{(v_{n})}\tilde{x}-%
\widehat{T}_{(u_{n})}\tilde{y}\Vert =\limsup_{(v_{n})\in \mathcal{D}}\Vert
\widehat{T}_{(v_{n})}\widehat{T}_{(u_{n})}\tilde{x}-\widehat{T}_{(u_{n})}%
\tilde{y}\Vert .
\end{equation*}
\end{lemma}

\begin{proof}
Fix $(u_{n})\in \mathcal{D},$ $\tilde{x},\tilde{y}\in \widetilde{C},$ and
notice that for every $(v_{n})\in \mathcal{D},$%
\begin{equation*}
\sup_{(\bar{w}_{n})\succeq (w_{n})}\Vert \widehat{T}_{(\bar{w}_{n})}\tilde{x}%
-\widehat{T}_{(u_{n})}\tilde{y}\Vert =\sup_{(\bar{v}_{n})\succeq
(v_{n})}\Vert \widehat{T}_{(\bar{v}_{n})}\widehat{T}_{(u_{n})}\tilde{x}-%
\widehat{T}_{(u_{n})}\tilde{y}\Vert ,
\end{equation*}%
where $(w_{n})$ is defined by (\ref{w_n}). Hence%
\begin{equation*}
\limsup_{(v_{n})\in \mathcal{D}}\Vert \widehat{T}_{(v_{n})}\tilde{x}-%
\widehat{T}_{(u_{n})}\tilde{y}\Vert \leq \limsup_{(v_{n})\in \mathcal{D}%
}\Vert \widehat{T}_{(v_{n})}\widehat{T}_{(u_{n})}\tilde{x}-\widehat{T}%
_{(u_{n})}\tilde{y}\Vert .
\end{equation*}%
The reverse inequality is obvious since $\widehat{T}_{(v_{n})}\widehat{T}%
_{(u_{n})}\tilde{x}=\widehat{T}_{(w_{n})}\tilde{x}$ and $(v_{n})\preceq
(w_{n}).$
\end{proof}

We are now in a\ position to prove the following generalization of Theorem %
\ref{Th1}.

\begin{theorem}
\label{Th21}Suppose $C$ is a nonempty bounded closed and convex subset of a
Banach space $X$ with SFPP and $\mathcal{T}=\left\{ T_{t}:t\in T\right\} $
is a commuting family of asymptotically nonexpansive mappings in the
intermediate sense acting on $C$. Then there exists $x\in C$ such that $%
T_{t}x=x$ for every $t\in T$ (a common fixed point for $\mathcal{T}$).
\end{theorem}

\begin{proof}
We partly follow the reasoning given in the proof of Theorem \ref{Th1}.
Assume that $X$ has the super fixed point property for nonexpansive
mappings. Let $\mathcal{T}=\left\{ T_{t}:t\in T\right\} $ be a commuting
family of asymptotically nonexpansive mappings in the intermediate sense
acting on a nonempty bounded closed and convex set $C\subset X.$ We can
assume that $\diam C=1.$ Take a free ultrafilter $p$ on $\mathbb{N},$ $%
(v_{n})\in \mathcal{D}$ and define $\widehat{T}_{(v_{n})}$ by (\ref{def2}).
Notice that for every $(v_{n}),(u_{n})\in \mathcal{D}$ and any $%
(x_{n})_{p}\in $ $\widetilde{C}$,
\begin{equation*}
(\widehat{T}_{(v_{n})}\circ \widehat{T}_{(u_{n})})(x_{n})_{p}=(\widehat{T}%
_{(u_{n})}\circ \widehat{T}_{(v_{n})})(x_{n})_{p}.
\end{equation*}%
It follows from Bruck's theorem that there exists $\widetilde{y}_{0}\in
\widetilde{C}$ such that $\widehat{T}_{(v_{n})}\widetilde{y}_{0}=\widetilde{y%
}_{0}$ for all $(v_{n})\in \mathcal{D}.$ Fix $\varepsilon <1$ and $x_{0}\in
C.$ We shall define by induction a sequence $(n_{j})$ of natural numbers and
a sequence $(x_{j})$ of elements in $C$ such that
\begin{equation}
\left\Vert x_{j}-x_{j-1}\right\Vert <3\varepsilon ^{j-1}\text{ and }\Vert
T_{u}x_{j}-x_{j}\Vert <\varepsilon ^{j}\text{ for every }u\in \mathcal{D}%
^{\prime }(n_{j}),j\geq 1.  \label{seq2}
\end{equation}%
By Lemma \ref{Lem2}, there exist $x_{1}\in C$ and $n_{1}\in \mathbb{N}$ such
that $\Vert T_{u}x_{1}-x_{1}\Vert <\varepsilon $ for every $u\in \mathcal{D}%
^{\prime }(n_{1})$ and $\left\Vert x_{1}-x_{0}\right\Vert \leq \diam C<3$.

Suppose that we have chosen\ natural numbers $n_{1},...,n_{j}$ and $%
x_{1},...,x_{j}\in C$ $(j\geq 1)$ such that%
\begin{equation}
\left\Vert x_{i}-x_{i-1}\right\Vert <3\varepsilon ^{i-1}\text{ and }\Vert
T_{u}x_{i}-x_{i}\Vert <\varepsilon ^{i}\text{ for every }u\in \mathcal{D}%
^{\prime }(n_{i}),1\leq i\leq j.  \label{ind}
\end{equation}

Let
\begin{equation*}
D_{j}=\left\{ \widetilde{y}=(y_{n})_{p}\in \widetilde{C}:\limsup_{(v_{n})\in
\mathcal{D}}\Vert \widehat{T}_{(v_{n})}\dot{x}_{j}-\widetilde{y}\Vert \leq
\varepsilon ^{j}\right\} ,
\end{equation*}%
where $\dot{x}_{j}=(x_{j},x_{j},...)_{p}$ and
\begin{equation*}
\limsup_{(v_{n})\in \mathcal{D}}\Vert \widehat{T}_{(v_{n})}\dot{x}_{j}-%
\widetilde{y}\Vert =\inf_{(v_{n})\in \mathcal{D}}\sup_{(u_{n})\succeq
(v_{n})}\lim_{n\rightarrow p}\left\Vert T_{u_{n}}x_{j}-y_{n}\right\Vert .
\end{equation*}%
Notice that for every $(v_{n})\in \mathcal{D}$ and $\eta >0,$ there exists $%
k\in \mathbb{N}$ such that $\sup_{x,y\in C}(\Vert T_{v_{n}}x-T_{v_{n}}y\Vert
-\Vert x-y\Vert )<\eta $ for every $n>k.$ Hence $v_{n}\in \mathcal{D}%
^{\prime }(n_{j})$ for sufficiently large $n$ and applying the induction
assumption (\ref{ind}) gives $\lim_{n\rightarrow p}\Vert
T_{v_{n}}x_{j}-x_{j}\Vert \leq \varepsilon ^{j}$ for every $(v_{n})\in
\mathcal{D}.$ It follows that $\dot{x}_{j}\in D_{j}$ and $D_{j}$ is a
nonempty closed and convex subset of $\widetilde{C}.$ By Lemma \ref{Lem3},
for fixed $(u_{n})\in \mathcal{D}$ and $\widetilde{y}\in D_{j},$%
\begin{align*}
& \limsup_{(v_{n})\in \mathcal{D}}\Vert \widehat{T}_{(v_{n})}\dot{x}_{j}-%
\widehat{T}_{(u_{n})}\widetilde{y}\Vert =\limsup_{(v_{n})\in \mathcal{D}%
}\Vert \widehat{T}_{(v_{n})}\widehat{T}_{(u_{n})}\dot{x}_{j}-\widehat{T}%
_{(u_{n})}\widetilde{y}\Vert \\
& \leq \limsup_{(v_{n})\in \mathcal{D}}\Vert \widehat{T}_{(v_{n})}\dot{x}%
_{j}-\widetilde{y}\Vert \leq \varepsilon ^{j},
\end{align*}%
and hence $\widehat{T}_{(u_{n})}(D_{j})\subset D_{j}$ for every $(u_{n})\in
\mathcal{D}.$ By Bruck's theorem, there exists $\widetilde{y}_{j}\in D_{j}$
such that $\widehat{T}_{(v_{n})}\widetilde{y}_{j}=\widetilde{y}_{j}$ for all
$(v_{n})\in \mathcal{D}.$ It is easy to see that $\left\Vert \widetilde{y}%
_{j}-\dot{x}_{j}\right\Vert <3\varepsilon ^{j}$ and, by Lemma \ref{Lem2},
there exist $x_{j+1}\in C$ and $n_{j+1}\in \mathbb{N}$ such that $\left\Vert
x_{j+1}-x_{j}\right\Vert <3\varepsilon ^{j}$ and $\Vert
T_{u}x_{j+1}-x_{j+1}\Vert <\varepsilon ^{j+1}$ for every $u\in \mathcal{D}%
^{\prime }(n_{j+1}).$ Thus we obtain by induction a sequence $(n_{j})$ of
natural numbers and a sequence $(x_{j})$ of elements in $C$ such that (\ref%
{seq2}) is satisfied. It follows that $(x_{j})$ is a Cauchy sequence
converging to some $x\in C$.

Fix $T_{t}\in \mathcal{T}$ and notice that for every $n_{j},$ there exists $%
k_{j}$ such that $\{(t,n)\}\in \mathcal{D}^{\prime }(n_{j})$ for $n>k_{j},$
since $T_{t}$ is asymptotically nonexpansive in the intermediate sense.
Applying (\ref{seq2}) gives%
\begin{equation*}
\limsup_{n\rightarrow \infty }\Vert T_{t}^{n}x_{j}-x_{j}\Vert \leq
\varepsilon ^{j},\ j\geq 1.
\end{equation*}%
Furthermore,%
\begin{align*}
\Vert T_{t}^{n}x-x\Vert & \leq (\Vert T_{t}^{n}x-T_{t}^{n}x_{j}\Vert -\Vert
x-x_{j}\Vert )+\Vert T_{t}^{n}x_{j}-x_{j}\Vert +2\Vert x_{j}-x\Vert \\
& \leq \sup_{x,y\in C}(\Vert T_{t}^{n}x-T_{t}^{n}y\Vert -\Vert x-y\Vert
)+\Vert T_{t}^{n}x_{j}-x_{j}\Vert +2\Vert x_{j}-x\Vert .
\end{align*}%
for every $j,n\geq 1$ and, consequently, $\limsup_{n\rightarrow \infty
}\Vert T_{t}^{n}x-x\Vert =0.$ Since $T_{t}$ is continuous, $T_{t}x=x$.
\end{proof}

\begin{remark}
It was proved in \cite[Th. 4]{DoLo} that if $C$ is a nonempty weakly compact
convex subset of a Banach space $X$ and every asymptotically nonexpansive
mapping of $C$ satisfies the $(\omega )$-fixed point property (which is a
little stronger than the fixed point property), then the set of common fixed
points of any commuting family of asymptotically nonexpansive mappings
acting on $C$ is a nonexpansive retract of $C.$ It is not known whether a
similar conclusion can be drawn under the assumptions of Theorem \ref{Th21}.
\end{remark}

\end{document}